\documentclass{amsart}
\usepackage{times, enumerate, array, longtable, %
amssymb, diagrams, mathrsfs, bbm, txfonts}
\usepackage{verbatim}
\usepackage[initials]{amsrefs}
\pdfoutput=1

\title{A Distributive Lattice Cover for Semilattices}
\author{Colin G. Bailey}
\address{School of Mathematics, Statistics and Operations Research\\
Victoria University of Wellington\\
Wellington, New Zealand\\
}
\email{Colin.Bailey@vuw.ac.nz}
\author{Joseph S.Oliveira}
\address{
Pacific Northwest National Laboratories\\
Richland,  WA\\
U.S.A.}
\email{Joseph.Oliveira@pnl.gov}
\date{\today}
\let\rsf\mathscr

\newcommand{\card}[1]{\left| #1\right|}
\def\one{\mathbf1}
\def\zero{\mathbf0}
\providecommand{\meet}{\mathbin{\wedge}}
\providecommand{\join}{\mathbin{\vee}}

\ifx\upharpoonright\undefined
     \def\restrict{\hbox{\rm\kern0.166em\accent"12\kern-0.536em$\vert$\kern0.3em}}%
  \else
     \def\restrict{\upharpoonright}%
  \fi
\makeatletter
\def\twoSet#1#2{\left\{%
\vphantom{#2}#1\thinspace\right|\nolinebreak[3]\left.%
  #2%
  \vphantom{#1}%
  \right\}%
}
\def\oneSet#1{\left\lbrace#1\right\rbrace}

\newif\if@nstr
\def\setstrfalse{\let\if@nstr=\iffalse}
\def\setstrtrue{\let\if@nstr=\iftrue}
\def\@nstr #1#2{
\def\@@nstr ##1#1##2##3\@@nstr{\ifx
\@nstr ##2\setstrfalse \else \setstrtrue \fi }
\@@nstr #2#1\@nstr \@@nstr}
\def\@separate#1|#2@{\setFront{#1}\setBack{#2}}
\def\lb#1\rb{\@nstr|{#1} \if@nstr \@separate#1 @ \twoSet{\@setFront}{\@setBack}%
\else \@separate |{#1 }@ \oneSet{\@setBack}\fi%
}
\def\setFront#1{\def\@setFront{#1}}
\def\setBack#1{\def\@setBack{#1}}
\def\Set#1{\lb{#1}\rb}

\def\oneBrk#1{\left\langle#1\right\rangle}
\def\twoBrk#1#2{\left\langle%
\vphantom{#2}#1\thinspace\right|\nolinebreak[3]\left.%
  #2%
  \vphantom{#1}%
  \right\rangle%
}
\def\brk<#1>{\@nstr|{#1} \if@nstr \@separate#1 @ \twoBrk{\@setFront}{\@setBack}%
\else \@separate |{#1 }@ \oneBrk{\@setBack}\fi%
}

\theoremstyle{plain}
\newtheorem{thm}{Theorem}[section]
\newtheorem{lem}[thm]{Lemma}
\newtheorem{cor}[thm]{Corollary}
\newtheorem{prop}[thm]{Proposition}
\newtheorem{defn}[thm]{Definition}

\theoremstyle{remark}
{}
{}
{}
{\newtheorem{rem}{Remark}[section]}

\@ifpackageloaded{bbm}{\newcommand{\Z}{{\mathbbm{Z}}}

}{
\newcommand{\Z}{{\mathbb{Z}}}

}
\makeatother

\begin{document}
    \begin{abstract}
        We consider two constructions of an envelope for a finite locally 
	distributive strong upper semilattice. The first is based on 
	Birkhoff's representation of finite distributive lattices and 
	the second on valuations on lattices. We show that these 
	produce isomorphic envelopes.
    \end{abstract}
\maketitle
\section{Introduction}

There are many interesting semilattices that also have meets whenever 
lower bounds exist -- see \cite{Abb,Abb:bk,Chajda:2005p59} for some 
implication based examples. For convenience we refer to these as strong upper 
semilattices. Thus these structures have the property that every 
interval $[a, b]$ is a lattice,  and if $[a, b]\subseteq[c, d]$ then 
$[a, b]$ is a sublattice of $[c, d]$. 

We consider those strong semilattices in which every interval is 
distributive. In this paper we concentrate primarily on finite 
examples and show that two well-known constructions lead to an 
\emph{envelope},  ie a distributive lattice in which our original 
structure embeds as an upwards-closed sub-semilattice. 

The first construction is based upon the Birkhoff representation 
theorem for finite distributive lattices,  and the second is based 
upon the valuation ring of a lattice. We show that these 
two constructions lead to the same (up to isomorphism) envelope. 


\begin{defn}
    A \emph{strong upper semilattice} is an upper semilattice with 
    one such that any two elements with a lower bound have a greatest 
    lower bound. As usual we will denote this glb by $x\meet y$. 
\end{defn}

\begin{defn}
    A \emph{distributive} strong upper semilattice is a strong upper 
    semilattice in which every interval $[a, b]$ is a distributive 
    lattice.
\end{defn}

\section{The Birkhoff Construction}
Throughout this section 
we assume that $\mathcal L$ is a \textbf{finite} distributive
strong upper semilattice. 
We intend to use the meet-irreducibles of $\mathcal L$ to get a 
distributive lattice as is done in Birkhoff's representation theorem 
for finite distributive lattices. Then we show that our original 
semilattice embeds into this new lattice as an upwards-closed 
subalgebra.

The first part of this proof verifies that the usual machinery for 
representation by meet-irreducibles works in our setting.  

Let 
$\rsf M$ be the set of meet-irreducibles of $\mathcal L$.
and if $x\in\mathcal L$ then $\rsf M^{x}=\Set{m\in\rsf M | x\le m}$.

Let $A=\Set{a\in\mathcal L | a\text{ is minimal}}$.

\begin{lem}\label{lem:wedgeMI}
    $$
	x=\bigwedge\rsf M^{x}.
    $$
\end{lem}
\begin{proof}
    By induction on co-rank.
    Everything of corank one is in $\rsf M$ and so the result is clear.
    
    Obviously $x\le\bigwedge\rsf M^{x}=s$. Suppose that $x<s$. Now if 
    $x=p\meet q$ for any $x< p, q$ in $\mathcal L$ then 
    $p=\bigwedge\rsf M^{p}$ and $q=\bigwedge\rsf M^{q}$ and $\rsf 
    M^{p}\cup\rsf M^{q}\subseteq\rsf M^{x}$ so that 
    $x= p\meet q= \bigwedge(\rsf M^{p}\cup\rsf 
    M^{q})\geq\bigwedge\rsf M^{x}>x$ -- contradiction. Hence $x$ must 
    be in $\rsf M$ and so the result is clear. 
\end{proof}

\begin{lem}
    Let $x\in\rsf M$ and $x^{+}=\bigwedge\Set{y\in\rsf M | x<y}$. Then
    \begin{enumerate}[i.]
	\item $x<x^{+}$; 
    
	\item if $z\in\rsf M$ is not equal to $x$ then $x\join 
	z=x^{+}\join z$.
	
	\item if $z\in\rsf M$ is not equal to $x$ then $x\join 
	z=x^{+}\join z^{+}$.
    \end{enumerate}
\end{lem}
\begin{proof}
    \begin{enumerate}[i.]
	\item This is immediate as $x$ is meet-irreducible.
    
	\item If $z\in\rsf M^{x}$ then $z\geq x^{+}$ and both sides 
	give $z$. So we may assume that $z\notin\rsf M^{x}$.
	
	Obviously $x\join z\le x^{+}\join z$. If $y\in\rsf M^{x\join z}$
    then $y>x$ and so $y\geq x^{+}$. Hence $y\geq x^{+}\join z$. 
    Thus $x\join z= \bigwedge\rsf M^{x\join z= \bigwedge\rsf 
    M^{x^{+}\join z}}= x^{+}\join z$. 
    
	\item $x\not=z$ gives $x\join z= x^{+}\join z= x\join z^{+}$. 
	Hence 
	$x^{+}\join z^{+}= (x\join x^{+})\join (z\join z^{+})= 
	(x\join z^{+})\join(x^{+}\join z)= x\join z$.
    \end{enumerate}    
\end{proof}

\subsection{Using Filters}

Now for the heart of the proof -- we let $\mathcal D=\mathcal D(\rsf M)$ be the set of order 
filters on $\brk<\rsf M, \le>$ with reverse inclusion -- so that 
$\Set{\one}$ is the greatest element and $\rsf M=\zero$ is the least. 
Union and intersection give meet and join. 

We define $\nu\colon\mathcal L\to\mathcal D$ by 
$$
    \nu(x)=\rsf M^{x}.
$$
This preserves order as $p\geq x\geq y$ and $p\in\rsf M$ implies 
$p\in\rsf M^{y}$,  ie $\rsf M^{x}\subseteq\rsf M^{y}$.

It is one-one as $x=\bigwedge\rsf M^{x}$.

It preserves join as $\rsf M^{x\join y}\subseteq\rsf M^{x}\cap\rsf 
M^{y}$ -- by the above. And if $p\geq x$ and $p\geq y$ then $p\geq 
x\join y$,  so we have equality. 

It preserves all meets that exist. Clearly we have $\rsf 
M^{x}\cup\rsf M^{y}\subseteq\rsf M^{x\meet y}$. 
And if $p\geq x\meet y$ is meet-irreducible, 
then $p=(p\join x)\meet(p\join y)= \bigwedge\rsf M^{p\join 
x}\meet\bigwedge\rsf M^{p\join y}$. If $p$ is in neither set then 
$\rsf M^{p\join x}\cup\rsf M^{p\join y}\subseteq\rsf M^{>p}$ and we 
know that $\bigwedge\rsf M^{>p}>p$ -- giving a  contradiction. Hence 
$p$ is in one of these two sets,  ie $p\geq x$ or $p\geq y$. 

We can extend these ideas, a little, to elements of $\mathcal D$. 
\begin{defn}
    Let $F\in\mathcal D$ and $a\in A$. Then 
    $f_{a}=\bigwedge(F\cap[a, \one])$.
\end{defn}

\begin{lem}
    Let $p\in\rsf M$, $F\in\mathcal D$ and $p\geq f_{a}(F)$. Then 
    $p\in F$.
\end{lem}
\begin{proof}
    We must have $\rsf M^{f_{a}(F)}= F\cap[a, \one]$ as $[a, \one]$ 
    is isomorphic to $\mathcal D([a, \one])$ and this describes that 
    isomorphism. 
\end{proof}

\begin{lem}
    Let $F, G\in\mathcal D$ and $a\in A$. Then
    \begin{align*}
	f_{a}(F\cup G) & = f_{a}(F)\meet f_{a}(G)  \\
	f_{a}(F\cap G) & = f_{a}(F)\join f_{a}(G) .
    \end{align*}
\end{lem}
\begin{proof}
    For the second we have $p\geq f_{a}(F)\join f_{a}(G)$ iff 
    $p\geq f_{a}(F)$ and $p\geq f_{a}(G)$ iff $p\in F\cap[a, \one]$ 
    and $p\in G\cap[a, \one]$ iff $p\in F\cap G\cap[a, \one]$ iff 
    $p\geq f_{a}(F\cap G)$. 
    
    For the first we have $p\geq f_{a}(F)\meet f_{a}(G)$ iff 
    $p\geq f_{a}(F)$ or $p\geq f_{a}(G)$ (as $x\mapsto\rsf M^{x}$ preserves 
    meets) iff $p\in F\cap[a, \one]$ or $p\in G\cap[a, \one]$ iff 
    $p\in(F\cap[a, \one])\cup(G\cap[a, \one])$ iff $p\in (F\cup 
    G)\cap[a, \one]$ iff $p\geq f_{a}(F\cup G)$.
\end{proof}


Lastly we observe that this construction gives a universal inclusion 
into a distributive lattice.
\begin{prop}
    Let $\phi\colon\mathcal L\to\mathcal S$ where $\mathcal S$ is any 
    distributive lattice,  $\phi$ preserves joins,  one and any 
    available meets. Then there is a unique lattice homomorphism 
    $\hat\phi\colon\mathcal D\to\mathcal S$ such that 
    $\phi= \hat\phi\nu$,  ie the following diagram commutes:
    \begin{diagram}
	    \mathcal L && \rTo^{\phi} && \mathcal S\\
	     & \rdTo_{\nu} && \ruTo_{\hat\phi}\\
	     && \mathcal D &&
	\end{diagram}
\end{prop}
\begin{proof}
    Define $\hat\phi\colon\mathcal D\to\mathcal S$ by 
    $$
	\hat\phi(F)=\bigwedge\phi[F].
    $$
    
    Firstly we see that $\hat\phi\nu(x)= \bigwedge\phi[\rsf M^{x}]= 
    \phi(\bigwedge\rsf M^{x})= \phi(x)$ (as $\phi$ preserves extant meets).
    
    It is clear that $\hat\phi$ preserves order. 
    
    $\hat\phi(F\cup G)= \bigwedge\phi[F\cup G]= 
    \bigwedge(\phi[F]\cup\phi[G])= 
    \bigwedge\phi[F]\meet\bigwedge\phi[G]= \hat\phi(F)\meet\hat\phi(G)$.    
    \begin{multline*}
    \hat\phi(F)\join\hat\phi(G)= 
    \bigwedge\phi[F]\join\bigwedge\phi[G]=
    \bigwedge_{\substack{f\in F\\g\in G}}\phi(f)\join\phi(g)=\\
    \bigwedge_{\substack{f\in F\\g\in G}}\phi(f\join g)= 
    \bigwedge_{\substack{f\in F\\g\in G}}\phi(\bigwedge\rsf M^{f\join 
    g})= 
    \bigwedge_{\substack{f\in F\\g\in G}}\phi(\rsf M^{f\join 
    g}).
\end{multline*}
Now 
    $\hat\phi(F\cap G)=\bigwedge_{p\in F\cap G}\phi(p)$ and the two 
    sets we're taking meets over are the same as if $p\in\rsf M$ then 
    $p\in F\cap G$ iff there are $f\in F$ and $g\in G$ with $p\geq 
    f\join g$. 
    
    Hence $\hat \phi$ is a lattice homomorphism. It might not 
    preserve $\zero$ as we do not know that $\bigwedge\phi[\rsf 
    M]=\zero_{S}$. 
\end{proof}

\section{The Valuation Ring}

Now we turn to another construction,  that of the valuation ring of a 
distributive lattice. We assume,  as above,  that $\mathcal L$ is a 
finite strong upper semilattice that is also distributive.

\begin{defn}
    For any abelian group $\rsf G$,  a function $f\colon\mathcal 
    L\to\rsf G$ is a \emph{valuation} iff for any $x, y, z\in L$
    $$
    f(x\join y\join z)+f((x\join z)\meet(y\join z))= f(x\join 
    z)+f(y\join z).
    $$
\end{defn}

This definition gives the usual definition if applied to a lattice  -- take 
$z=x\meet y$.

We now consider the construction of the valuation ring of $\mathcal 
L$. This parallels the development in \cite{BBR}.

Now let $\Z^{\mathcal L}$ be the free abelian group
over $\mathcal L$. This has a multiplication induced by
$\brk<x, y>\mapsto x\join y$ on $\mathcal L$ and extended by linearity. 

We can define a submodule
$M(\mathcal L)$ via generators $a\join b\join c+(a\join c)\meet 
(b\join c)-a\join c-b\join c$ for all $a, b, c\in\mathcal L$.
This is in fact an ideal of the ring when $\mathcal L$ is distributive -- as $t\in\mathcal L$ and 
$a\join b\join c+(a\join c)\meet 
(b\join c)-a\join c-b\join c\in M(\mathcal L)$ implies 
\begin{align*}
    (a\join b\join c+&(a\join c)\meet 
(b\join c)-a\join c-b\join c)\join t \\
&= (a\join b\join c\join 
t)+((a\join c)\meet(b\join c))\join t
- (a\join c\join t)-(b\join 
c\join t)\\
&= (a\join t)\join (b\join t)\join(c\join t)+((a\join t)\join (c\join 
t))\meet((b\join t)\join(c\join t))\\
&\qquad-(b\join t)\join(c\join t)-(a\join 
t)\join(c\join t)\\
&\in M(\mathcal L).
\end{align*}

On quotienting out by this subgroup we get the valuation ring 
$V(\mathcal L)$.

\begin{rem}
    In lattices,  this reduces to the same as the usual definition 
    for valuations.
\end{rem}

%
%
%
%
%

There is a natural mapping $\iota\colon\mathcal L\to V(\mathcal L)$
given by $x\mapsto x$. This is a valuation by the definition of 
$M(\mathcal L)$. Furthermore it is a universal valuation.
\begin{thm}
    Let $f\colon\mathcal L\to\mathcal G$ be any valuation. Then there is a 
    unique
    group homomorphism $\phi\colon V(\mathcal L)\to\mathcal G$ such 
    that the following diagram commutes:
    \begin{diagram}
	\mathcal L && \rTo^{f} && \mathcal G\\
	 & \rdTo_{\iota} && \ruTo_{\phi}\\
	 && V(\mathcal L) &&
    \end{diagram}
\end{thm}
\begin{proof}
    Define $\Phi\restrict\mathcal L$ to be $f$,  and extend linearly 
    to all of $\Z^{\mathcal L}$. Then $\Phi$ is zero on $M(\mathcal 
    L)$ -- as $f$ is a valuation -- 
    and so factors through $V(\mathcal L)$ to give $\phi$.
    
    Uniqueness is immediate from our definition -- we must have 
    agreement with $f$ and linearity.
\end{proof}

%

\begin{cor}
    Let $f\colon\mathcal L_{1}\to\mathcal L_{2}$ be a homomorphism of 
    suitable type (ie it preserves $\join, \ \one$ and such meets 
    as exist). Then there is a ring homomorphism $V(f)\colon V(\mathcal 
    L_{1})\to V(\mathcal L_{2})$ such that 
    $$
    \begin{diagram}
        \mathcal L_{1} & \rTo^{f} & \mathcal L_{2}  \\
        \dTo^{\iota_{1}} &  & \dTo_{\iota_{2}}  \\
        V(\mathcal L_{1}) & \rTo_{V(f)} & V(\mathcal L_{2})
    \end{diagram}
    $$
    commutes.
\end{cor}
\begin{proof}
    We note that $\iota_{2}f$ is a valuation on $\mathcal L_{1}$ and 
    so the theorem gives the function $V(f)$.
\end{proof}

This corollary shows that $V$ is functorial. It is not hard to see 
that $V(f)$ is actually a ring homomorphism.

We will assume that the reader knows the basics of valuation theory 
on lattices in what follows. A good reference is 
\cite{GeiVal1}.

Some properties of $\iota$ are now evident. As constant functions are 
valuations we see that if $g\in\mathcal G$ is any element then 
$c_{g}\colon\mathcal L\to\mathcal G$ is a valuation and so for any 
$l\in\mathcal L$ we have that the order of $\iota(l)$
is a multiple of the order of $\phi_{c}\iota(l)=g$. Hence $\iota(l)$ 
has infinite order. 
 
We also see that if $\mathcal L$ is a distributive lattice then any 
two elements may be separated by a prime filter $\rsf F$. The 
characteristic function of $\rsf F$ into $\Z_{2}$ is a valuation and 
the two elements under consideration have distinct values. Hence they 
must be separated by $\iota$. 

\begin{prop}\label{prop:oneOne}
    Let $\mathcal L$ be a strong upper semilattice with $\one$ that 
    is locally principal,  ie there is a set $A$ of elements that are 
    all minimal in $\mathcal L$ such that $\mathcal L=\bigcup_{a\in 
    A}[a, \one]$. 
    
    Then $\iota$ is one-one on $\mathcal L$.
\end{prop}
\begin{proof}
    The proof proceeds by considering the distributive lattices $[a, 
    \one]$ for $a\in A$ and the mappings $\iota_{a}\colon[a, \one]\to 
    V([a, \one])$,  which we already know to be one-one. 
    
    Now we have $\Z^{([a, \one])}\subseteq\Z^{(\mathcal L)}$ and if
    $t$ is a generator of $M([a, \one])$ then it is also in 
    $M(\mathcal L)$. Furthermore if $s$ is a generator of $M(\mathcal 
    L)$ with $s\in \Z^{([a, \one])}$ then the freeness of the module 
    implies $s$ is a generator for $M([a, \one])$. Thus we have 
    $$
	M([a, \one])=M(\mathcal L)\cap\Z^{([a, 1])}.
    $$
    This means we have an injective homomorphism $\vartheta_{a}\colon 
    V([a, \one])\to V(\mathcal L)$ defined by 
    $$
	\vartheta_{a}([x]_{a})=[x]_{\mathcal L}
    $$
    (the subscripts indicate where the equivalance class lives).
    We have the following properties of $\vartheta_{a}$:
    \begin{description}
        \item[Well-defined \& one-one] as $c\sim_{a}d$ iff $c-d\in 
	M([a, \one])$ iff $c-d\in M(\mathcal L)\cap\Z^{([a, \one])}$ 
	iff $c\sim_{\mathcal L}d$.
    
        \item[Preserves $+$] as the operation $+$ is the same in 
	$\Z^{([a, \one])}$ as in $\Z^{(\mathcal L)}$.
    
        \item[Preserves $\times$] as $[a, \one]$ is a sub-semilattice 
	of $\mathcal L$ and so multiplication in $\Z^{([a, \one])}$
	is the same as in $\Z^{(\mathcal L)}$.
    \end{description}
    
    Now define a function $j_{a}\colon\mathcal L\to V([a, \one])$ by
    $j_{a}(y)=\iota_{a}(y\join a)$. Then $j_{a}$ is a valuation -- 
    as 
    \begin{align*}
        j_{a}(x\join y\join z)&+j_{a}((x\join z)\meet(y\join z)) \\
	&= \iota_{a}(x\join y\join z\join a)+\iota_{a}([(x\join 
	z)\meet(y\join z)]\join a)\\
         & = \iota_{a}((x\join a)\join(y\join a)\join(z\join a))+
	 \iota_{a}((x\join z\join a)\meet(y\join z\join a))\\
         & = \iota_{a}((x\join a)\join(y\join a)\join(z\join a))+
	 \iota_{a}((x\join a)\join (z\join a))\meet((y\join a)\join (z\join a)))\\
	 & = \iota_{a}((x\join a)\join(z\join a))+\iota_{a}((y\join 
	 a)\join(z\join a))\\
	 &= \iota_{a}(x\join z\join a)+\iota_{a}(y\join z\join a)\\
	 &= j_{a}(x\join z)+j_{a}(y\join z).
    \end{align*}
    
    Suppose that $\iota_{\mathcal L}(x)=\iota_{\mathcal L}(y)$.
    
    Let $\varphi_{a}\colon V(\mathcal L)\to V([a, \one])$ be such that 
    $j_{a}=\varphi_{a}\iota_{\mathcal L}$. Then $\iota_{\mathcal L}(x)=\iota_{\mathcal L}(y)$
    implies $j_{a}(x)=j_{a}(y)$ and so $\iota_{a}(x)= \iota_{a}(x\join 
    a)= \iota_{a}(y\join a)$. As $\iota_{a}$ is one-one we have 
    $x=y\join a$. Likewise $y=x\join b$. And so 
    $y= x\join b= y\join a\join b= y\join a\geq a$. (and $x\geq b$.)
    
    But now $\vartheta_{a}\iota_{a}(x)= \iota_{\mathcal L}(x)= 
    \iota_{\mathcal L}(y)= \vartheta_{a}\iota_{a}(y)$ implies
    $x=y$ -- as both $\vartheta_{a}$ and $\iota_{a}$ are one-one.
\end{proof}

In this result we have a lattice homomorphism $\text{incl}\colon[a, 
\one]\to\mathcal L$ and it's clear that $\vartheta_{a}=V(\text{incl})$. 
Because $\mathcal L$ is distributive,  the mapping 
$J_{a}\colon x\mapsto x\join a$ is a lattice homomorphism from $\mathcal L\to[a, 
\one]$ and $\varphi_{a}=V(J_{a})$. Also $\text{incl}_{\circ} 
J_{a}=\text{id}_{a}$ and so 
$\varphi_{a}\vartheta_{a}=V(\text{id}_{a})=\text{id}$ and so we 
actually have a retract. This gives a kind of localization on 
$V(\mathcal L)$ that is extremely useful.

\subsection{The envelope}
Now that we have the valuation ring we can define a new partial 
operation that produces our envelope. 

We begin by defining
$$
x\meet y = x+y-x\join y
$$
and extending for any set $X\subseteq \mathcal L$ by 
letting $X_{(m)}=\Set{x | x\text{ is minimal in }X}$.
$$
\bigwedge X= \sum_{\substack{B\subset X_{(m)}\\B\not=X_{(m)}}}(-1)^{\card 
B+1}\bigvee B.
$$
The envelope is the range of this operation. We want to show that 
this indeed produces a distributive lattice.

Rather than go through the tedious calculations involved we consider 
the relationship between the valuation ring of $\mathcal L$ and the 
valuation ring of the distributive envelope constructed above,  ie 
$V(\mathcal D)$. 

\section{The two envelopes are equivalent}
We want to consider how to extend valuations on $\mathcal L$ to 
valuations on $\mathcal D$. Of course we need only extend the 
universal valuation $\iota$. The most natural method is to equate a 
filter $F\in\mathcal D$ with its ``virtual meet'' $\bigwedge F= 
\bigwedge_{a\in A}f_{a}(F)$ and 
use inclusion-exclusion to define
$$
i(F)=\sum_{\substack{B\subseteq A\\
\card B\geq 1}}(-1)^{\card B+1}\bigvee_{b\in B}f_{b}(F).
$$
However it might be simpler to just use all of $F$ and define it as 
$$
j(F)=\sum_{\substack{H\subseteq F\\
\card H\geq 1}}(-1)^{\card H+1}\bigvee H.
$$
The calculations involved in checking this are not easy and so we 
argue indirectly that this method must work.

We know that valuations on a distributive 
lattice are completely 
determined by their values on the set of meet-irreducibles and $\one$ 
(because these form a basis).
Furthermore we know that (for distributive lattices) that we can 
define a valuation from any set of values on the meet-irreducibles 
and $\one$ (same reason). 
Putting these together means that we can start with a valuation on 
$\mathcal L$,  lift to one on $\mathcal D$ by using its values on 
$\rsf M(\mathcal L)= \rsf M(\mathcal D)$,  and restrict to $\mathcal 
L$. 

Now the inclusion/exclusion lemma holds true in both $V(\mathcal L)$
and in $V(\mathcal D)$.

\begin{lem}\label{lem:ie}
    Let $x_{1}, \dots, x_{n}\in\mathcal L$ be such that 
    $\bigwedge_{i}x_{i}$ exists. Then
    $$
    \bigwedge_{i}x_{i}=\sum_{\substack{J\subseteq\Set{1, \dots, i}\\
    \card J>0}}(-1)^{\card J+1}\bigvee_{j\in J}x_{j}.
    $$
\end{lem}
\begin{proof}
    Let $a\le\bigwedge_{i}x_{i}$ be minimal. Then it holds in $V([a, 
    \one])$ which is a subalgebra of $V(\mathcal L)$ and so it holds 
    in $V(\mathcal L)$.
\end{proof}

By this inclusion/exclusion formula in $V(\mathcal L)$ the 
restriction must be same as the valuation we started 
with. This works both ways:
as if $f$ is a valuation on $\mathcal D$ then it restricts to one on 
$\mathcal L$.
Thus the valuation rings must be the 
same -- ie this mapping makes both $V(\mathcal D)$ and $V(\mathcal 
L)$ universal valuations for $\mathcal D$ and so they are isomorphic. 

Thus either of the definitions above must work,  and be the same. 
In particular,  if $X\subseteq\mathcal L$ and $F_{X}$ is the filter 
generated by $X$ then 
$\bigwedge X= F_{X}$ in $\mathcal D$ and 
we see that $X$ and $F_{X}$ have the same set of minimal elements 
$X_{(m)}$. Therefore $\bigwedge X_{(m)}= F_{X}$ in $\mathcal D$ also.
Then
$$
    \bigwedge X = 
    \sum_{\substack{B\subset X_{(m)}\\B\not=X_{(m)}}}
    (-1)^{\card B+1}\bigvee B  
      = \bigwedge F_{X}  
      = F_{X}.
$$
Thus the envelope we defined above is exactly $\mathcal D$. This 
gives us our 
alternative representation for $\mathcal D$.

%
%
%
%
%

\end{document}